\documentclass[a4paper,draft]{amsart}
\usepackage[a4paper,margin=25mm]{geometry}
\usepackage[utf8]{inputenc}
\usepackage{mathtools}
\usepackage{amsfonts}
\usepackage{amsthm}
\usepackage{amsmath}
\usepackage{amssymb}
\usepackage{tikz}
\usepackage{hyperref}
\usepackage{mathrsfs}

\newtheorem{Problem}{Problem}

\newtheorem{Theorem}{Theorem}
\newtheorem{Proposition}{Proposition}

\newtheorem{Conjecture}{Conjecture}
\newtheorem{Corollary}{Corollary}
\newtheorem{Algorithm}{Algorithm}

\def\eatspace#1{#1}
\def\step#1#2{\par\kern1pt\dimen44=#2em\advance\dimen44 1.67em\hangindent\dimen44\hangafter=1\noindent\rlap{\small#1}\kern\dimen44\relax\eatspace}
\let\set\mathbb
\def\<#1>{\langle#1\rangle}

\theoremstyle{definition}
\newtheorem{Definition}{Definition}
\newtheorem{Example}{Example}

\makeatletter
\def\testb#1{\testb@i#1,,\@nil}%
\def\testb@i#1,#2,#3\@nil{%
  \draw[->] (O) --++(#1);
  \ifx\relax#2\relax\else\testb@i#2,#3\@nil\fi}
\makeatother


\begin{document}

\title{The Newton-Puiseux Algorithm and Effective Algebraic Series}

\maketitle

\begin{center}
\begin{tabular}{@{}c@{}}
    Manfred Buchacher \\
    Johannes Kepler Universit{\"a}t Linz\\
    \normalsize manfredi.buchacher@gmail.com
  \end{tabular}%
  \end{center}

\begin{abstract}
We explain how to encode an algebraic series by finite data
and how to do effective arithmetic on the level of these encodings. The reasoning is based on the Newton-Puiseux algorithm and an effective equality test for algebraic series. Furthermore, we discuss how to derive information about the support of an algebraic series.
Based thereon, we show how to identify the polynomial and rational solutions of a polynomial equation.  
\end{abstract}

\section{Overview}

Given a polynomial equation in two unknowns over an algebraically closed field of characteristic zero, the Newton-Puiseux algorithm~\cite{puiseux1850recherches,brieskorn2012plane} solves the equation for one of the unknowns in terms of series in the other by computing them term by term. 
Finding a series solution of a polynomial equation is one and the most apparent aspect of the algorithm. However, it also permits to encode algebraic series by finite data and to effectively compute with them on the level of these encodings. While this is well-known for univariate algebraic series, this is not the case for algebraic series that are multivariate. We explain how to do effective arithmetic with such series and complement the discussion of the Newton-Puiseux algorithm for multivariate, not necessarily bivariate polynomials over a field of characteristic zero in~\cite{MacDonald}. We also explain that the convex hull of the support of an algebraic series is a polyhedral set and discuss how its vertices and bounded faces can be computed. Based thereon, we provide a simple way to identify the polynomial and rational solutions of a polynomial equation. The article comes with a Mathematica implementation of the Newton-Puiseux algorithm and a Mathematica notebook that contains the examples presented here. They can be found on~https://github.com/buchacm/newtonPuiseuxAlgorithm. 

Further literature related to the Newton-Puiseux algorithm, effective algebraic series or supports of algebraic series: a generalization of the Newton-Puiseux algorithm to systems of polynomial equations over a field of characteristic zero is presented in~\cite{mcdonald2002fractional,aroca2010puiseux, aroca2009family}, see also the references therein, and to polynomials over a field of positive characteristic in~\cite{saavedra2017mcdonald}. Supports of series were also studied in~\cite{aroca2022minimal,aroca2019support}, though in the more general context of series algebraic over a certain ring of series. For a complexity analysis of (variants of) the Newton-Puiseux algorithm we refer to~\cite{walsh2000polynomial}, in particular to~\cite{poteaux2015improving} and the references therein. We point out that there are other computational models for algebraic (power) series based on the implicit function theorem~\cite{alonso1992computational, alonso2018encoding}, diagonals of rational functions~\cite{denef1987algebraic}, or closed-form formulas of their coefficients~\cite{hickel2019algebraic}.

\section{Preliminaries}\label{sec:prelim}

We begin with introducing the objects this article is about: multivariate algebraic series and the Newton-Puiseux algorithm to constructively work with them.

We use multi-index notation. We denote by $\bold{x}=(x_1,\dots,x_n)$ a vector of variables, and we write $\bold{x}^I=x_1^{i_1} \cdots x_n^{i^n}$ for the monomial whose exponent vector is $I = (i_1,\dots,i_n)\in\mathbb{Q}^n$. We work over a (computable) algebraically closed field $\mathbb{K}$ of characteristic zero. A \emph{series} $\phi$ in $\bold{x}$ over~$\mathbb{K}$ is a formal sum 
\begin{equation*}
\phi = \sum_{I\in\mathbb{Q}^n} a_I \bold{x}^I
\end{equation*}
of terms in $\bold{x}$ whose coefficients $a_I$ are elements of $\mathbb{K}$.
Its \emph{support} is defined by 
\begin{equation*}
\mathrm{supp}(\phi) = \{I\in\mathbb{Q}^n: a_I \neq 0 \},
\end{equation*}
and we will assume throughout that there is a vector~$v\in\mathbb{R}^n$, a strictly convex rational cone $C\subseteq\mathbb{R}^n$ and an integer $k\in\mathbb{Z}$ such that 
\begin{equation*}
\mathrm{supp}(\phi) \subseteq \left( v + C \right) \cap \frac{1}{k}\mathbb{Z}^n.
\end{equation*}

We recall that a \emph{cone} $C$ is a subset of $\mathbb{R}^n$ that is closed under multiplication with non-negative numbers. It is called \emph{convex}, if for any two points of $C$, it also contains the line segment that joins them, and it is said to be \emph{strictly convex}, if it is convex and does not contain any lines. It is \emph{rational}, if there are $v_1,\dots,v_k \in\mathbb{Z}^n$ such that 
\begin{equation*}
C = \mathrm{cone}\{v_1,\dots,v_k\} := \mathbb{R}_{\geq 0}\cdot v_1 + \dots + \mathbb{R}_{\geq0}\cdot v_k.
\end{equation*} 
Although not always explicitly stated, all cones appearing in this text are strictly convex rational cones. 

Without any restriction on their supports, the sum and product of two series is not well-defined.
\begin{Example}\label{ex:well-definedness}
The geometric series
\begin{equation*}
\phi_1 = 1 + x + x ^2 + \dots \quad \text{and} \quad \phi_2 = - x^{-1} - x^{-2} - x^{-3} - \dots
\end{equation*}
are series in the above sense, but neither is their sum nor their product. Their product is not meaningful, since its coefficients involve infinite sums, and their sum is not well-defined, because its support is not contained in a shift of a strictly convex cone.

\end{Example}
For any strictly convex rational cone $C\subseteq\mathbb{R}^n$ the set~$\mathbb{K}_C[[x]]$ of series whose support is contained in $C$ is a ring with respect to addition and multiplication~\cite[Theorem 10]{monforte2013formal}. Yet, it is not a field~\cite[Theorem 12]{monforte2013formal}. Any~$w\in\mathbb{R}^n$ whose components are linearly independent over $\mathbb{Q}$ defines an additive total order $\preceq$ on~$\mathbb{Q}^n$ by
\begin{equation*}
\alpha \preceq \beta \quad :\Longleftrightarrow \quad \alpha \cdot w \leq  \beta \cdot w,
\end{equation*}
where $v \cdot w$ denotes the Euclidean inner product of $v,w\in\mathbb{R}^n$.
A cone $C\subseteq\mathbb{R}^n$ is \emph{compatible} with $w\in\mathbb{R}^n$ and the induced total order~$\preceq$ on $\mathbb{Q}^n$, if $C\cap \mathbb{Q}^n$ has a maximal element.
\begin{figure}
\begin{center}
  \begin{tikzpicture}[scale=.2]
    \draw[->] (-6.5,-3)--(6.5,-3);
    \draw[->](-3,-6.5)--(-3,6.5);
    \draw (-29/18,-5.5)--(-5.5,1.5); 
    \draw[->] (-3,-3)--(-24/5,-4) node[left] {$w$};
    \draw (-5.5,0) node {$\bullet$} node[below] {$\alpha$};
    \draw (-2,1.25) node {$\bullet$} node[below] {$\beta$} ;     
    \draw (-1.5,2.5) node {$\bullet$} node[right] {$\gamma$} ;
    \draw (-271/72,11/2)--(169/72,-11/2);
     \begin{scope}[xshift=20cm]
     \fill[lightgray] (-3,-3)--(-5.5,5.5)--(5.5,5.5)--(5.5,-5.5)--cycle;
    \fill[gray] (-3,-3)--(-1,5.5)--(5.5,5.5)--(5.5,-1)--cycle;
    \draw[->](-6.5,-3)--(6.5,-3);
    \draw[->](-3,-6.5)--(-3,6.5);
    \draw (-29/18,-5.5)--(-5.5,1.5); 
    \draw[->] (-3,-3)--(-24/5,-4) node[left] {$w$};
\end{scope}
  \end{tikzpicture}
\end{center}
 \caption{ A vector $w\in\mathbb{R}^2$ that induces an additive total order $\preceq$ on $\mathbb{Q}^2$ such that $\alpha > \beta > \gamma$ and two cones that are compatible with it.}
\end{figure}
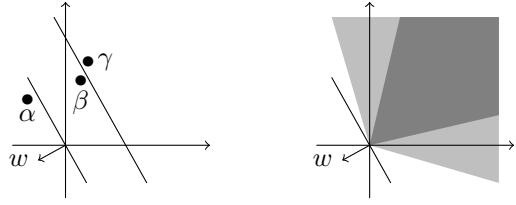
Given an additive total order~$\preceq$ on $\mathbb{Q}^n$, we denote by $\mathscr{C}$ the set of compatible strictly convex rational cones, and we write
\begin{equation*}
\mathbb{K}_{\preceq}((\bold{x})):= \bigcup_{C\in\mathscr{C}} \bigcup_{e\in\set Q^n} x^e \mathbb{K}_C[[\bold{x}]]
\end{equation*} 
for the set of series whose support is contained in a shift of such a cone. It is not only a ring but even a field~\cite[Theorem 15]{monforte2013formal}.

A series $\phi$ is said to be \emph{algebraic}, if there is a non-zero polynomial~$p\in\mathbb{K}[\bold{x},y]$ such that 
\begin{equation*}
p(\bold{x},\phi) = 0,
\end{equation*}
and it is said to be \emph{D-finite}~\cite{lipshitz1989d}, if for every $i\in\{1,\dots,n\}$ there are polynomials $q_0,\dots,q_r\in\mathbb{K}[\bold{x}]$ such that 
\begin{equation*}
q_0 \phi + q_1 \frac{\partial}{\partial x_i}\phi + \dots + q_r \frac{\partial^r}{\partial x_i^r}\phi = 0.
\end{equation*}
Every algebraic series is D-finite~\cite[Theorem 6.1]{kauers2011concrete}, and just as algebraic series, D-finite series satisfy many closure properties. For instance, the sum $\phi_1 + \phi_2$ of two D-finite series $\phi_1$ and $\phi_2$ is D-finite~\cite[Theorem 7.2]{kauers2011concrete}, and so is the restriction of a D-finite series $\phi$ to a rational cone $C$~\cite{bostan2017hypergeometric}, i.e. 
\begin{equation*}
[\phi]_C(\bold{x}) := \sum_{I\in C\cap \mathbb{Q}^n}\left( [\bold{x}^I]\phi\right) \bold{x}^I,
\end{equation*}
where $[\bold{x}^I]\phi$ denotes the coefficient of $\bold{x}^I$ in $\phi$. These closure properties are effective in the sense that systems of differential equations for $\phi_1 + \phi_2$ and $[\phi]_C$ can be computed from the differential equations satisfied by $\phi_1$, $\phi_2$ and $\phi$. 

Having a univariate D-finite series $\phi(t) := \sum_{k\geq k_0} \phi_k t^k$ in $t$ and a differential equation it satisfies, it can be checked effectively whether it is identically zero. The differential equation for $\phi$ translates into a recurrence relation for its coefficients,
\begin{equation*}
p_0(k) \phi_k + p_1(k)\phi_{k+1} + \dots + p_r(k)\phi_{k+r} = 0, \quad p_0,\dots,p_r \in\mathbb{K}[k].
\end{equation*}
It follows that $\phi = 0$ if and only if $\phi_k = 0$ for finitely many $k$, the number of terms which have to be compared to zero depending on the order of the recurrence and the largest integer root of $p_r$.

Given $p\in\mathbb{K}[\bold{x},y]$ and a total order $\preceq$ on $\mathbb{Q}^n$ induced by $w\in\mathbb{R}^n$, the Newton-Puiseux algorithm~\cite{MacDonald} determines the first terms of the series solutions of $p(\bold{x},y) = 0$ in~$\mathbb{K}_\preceq((\bold{x}))$.
We collect some definitions before its presentation. The \emph{Newton polytope} of~$p$ is the convex hull of the support of $p$,
 \begin{equation*}
\mathrm{Newt}(p) := \mathrm{conv}(\mathrm{supp}(p)).
\end{equation*}
If $e$ is an edge of $\mathrm{Newt}(p)$ that connects two vertices $v_1$ and $v_2$, we write $e = \{v_1,v_2\}$. It is called \emph{admissible}, if $v_{1,n+1}\neq v_{2,n+1}$. If $v_{1,n+1} < v_{2,n+1}$, we call $v_1$ and $v_2$ the minor and major vertex of~$e$, respectively, and denote them by~$\mathrm{m}(e)$ and $\mathrm{M}(e)$. An \emph{edge path} on $\mathrm{Newt}(p)$ is a sequence $e_1,\dots,e_k$ of edges such that $\mathrm{m}(e_{i+1}) = \mathrm{M}(e_i)$ for $i=1,\dots,k-1$. Let $P_e$ be the projection on $\mathbb{R}^{n+1}$ that projects on~$\mathbb{R}^n\times \{0\}$ along lines parallel to an (admissible) edge $e$. The \emph{barrier cone} of~$e$ is the smallest cone that contains $P_e(\mathrm{Newt}(p)) - P_e(e)$, the projection of the Newton polytope of $p$ shifted by the projection of its edge $e$. It is denoted by $C(e)$. We occasionally identify $\mathbb{R}^n$ and $\mathbb{R}^n\times \{0\}\subseteq \mathbb{R}^{n+1}$, and consider $C(e)$ as a subset of $\mathbb{R}^n$. Its \emph{dual} is
\begin{equation*}
C(e)^* := \{ w\in \mathbb{R}^n : v\cdot w \leq 0 \text{ for all } v \in C(e) \}.
\end{equation*}
A vector $w\in\mathbb{R}^n$ that defines a total order on $\mathbb{Q}^n$ is said to be \emph{compatible} with $e$, if $w\in C(e)^*$, and an edge path is called \emph{coherent}, if there is a $w\in\mathbb{R}^n$ that is compatible with all of its edges.

\begin{figure}
\begin{center}
  \begin{tikzpicture}[scale=.2]
    \fill[gray] (0,0)--(1,5.5)--(5.5,5.5)--(5.5,1)--cycle;
    \fill[lightgray] (0,0)--(-5.5,1)--(-5.5,-5.5)--(1,-5.5)--cycle;
    \draw[->](-6.5,0)--(6.5,0);
    \draw[->](0,-6.5)--(0,6.5);
 \begin{scope}[xshift=20cm]
   \fill[lightgray] (-5.5,-0.5)--(-5.5,5.5)--(5.5,5.5)--(5.5,-5.5)--(-0.5,-5.5)--cycle;
    \fill[gray] (0,-1)--(5.5,-4)--(5.5,5.5)--(0,5.5)--cycle;
    \draw[->](-6.5,-3)--(6.5,-3);
    \draw[->](-3,-6.5)--(-3,6.5);
    \foreach \x/\y in {-3/-2, -3/-1, -3/-4, -1/0, -4/-1, -4/-2, -2/-2, -2/-4.5, -2/-5.5, -1/-3, 0/-3, 0/-1} \draw (\x,\y) node {$\bullet$};
    \end{scope}
  \end{tikzpicture}
\end{center}
  \caption{Two strictly convex cones that are dual to each other, and the support of the first terms of a series and a shifted cone that contains the remaining (non-depicted) support.}
\end{figure}
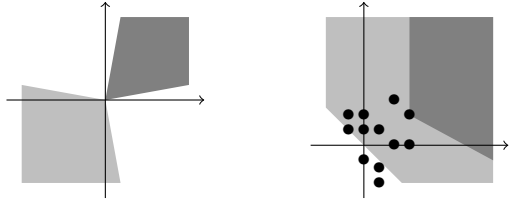
Given an admissible edge $e = \{v_1,v_2\}$, we denote its \emph{slope} with respect to its last coordinate by
\begin{equation*}
\mathrm{S}(e):= \frac{1}{v_{2,n+1}-v_{1,n+1}}(v_{2,1}-v_{1,1},\dots,v_{2,n}-v_{1,n}).
\end{equation*}
The edge polynomial $p_e(t)$ of an edge $e$ of the Newton polytope of $p$ is 
\begin{equation*}
p_e(t)=\sum_{I} a_I t^{I_{n+1}-\mathrm{m}(e)_{n+1}},
\end{equation*}
where $a_I = [(\bold{x},y)^I]p$ and the sum runs over all $I$ in~$e\cap \mathrm{supp}(p)$.

We now present the Newton-Puiseux algorithm. For details, in particular for a proof of its correctness, we refer to~\cite[Theorem~3.5]{MacDonald}.

\begin{Algorithm}[Newton-Puiseux Algorithm]\label{alg:NPA}
Input: A square-free and non-constant polynomial $p\in\mathbb{K}[\bold{x},y]$, an admissible edge $e$ of its Newton polytope, an element $w$ of the dual of its barrier cone $C(e)$ defining a total order on $\mathbb{Q}^n$, and a (non-negative) integer $k$.\\
  Output: A list of $\mathrm{M}(e)_{n+1}-\mathrm{m}(e)_{n+1}$ many pairs $(c_1\bold{x}^{\alpha_1}+\dots+c_N \bold{x}^{\alpha_N},C)$ with $c_1\bold{x}^{\alpha_1},\dots,c_N\bold{x}^{\alpha_N}$ being the first $N$ terms of a series solution $\phi$ of $p(\bold{x},\phi) = 0$, ordered with respect to $w$ in decreasing order, and $C$ being a strictly convex rational cone such that $\mathrm{supp}(\phi)\subseteq \{\alpha_1,\dots,\alpha_{N-1}\} \cup \left(\alpha_N + C\right)$, where $N$ is the smallest integer greater than or equal to $k$ such that the series solutions can be distinguished by their first $N$ terms.
  \step 10 Compute the roots $c$ of the edge polynomial $p_e(t)$ associated with $e$, let $L_1$ be the set of (pairwise different) pairs~$(\phi,e)$ with $\phi = c \bold{x}^{-\mathrm{S}(e)}$ and set $L_{\mathrm{finite}}=\emptyset$ and $N=1$.
  \step 20 While $|L_{\mathrm{finite}}| +  |L_1| < \mathrm{M}(e)_{n+1}-\mathrm{m}(e)_{n+1}$ or $\left(|L_{\mathrm{finite}}| <  \mathrm{M}(e)_{n+1}-\mathrm{m}(e)_{n+1} \text{ and } N< k\right)$, do:
  \step 31 Set $L_2 = \emptyset$ and $N = N+1$.
  \step 41 While there is a $(\phi,e)\in L_1$ with $\phi$ not having $k$ terms or $p_e(t)$ not having only simple roots, do for each $(\phi,e)\in L_1$:
  \step 52 If $\phi$ satisfies $p(\bold{x},\phi)=0$, append $(\phi,e)$ to $L_{\mathrm{finite}}$, otherwise compute the Newton polytope of $p(\bold{x},\phi+y)$ and determine its unique edge path $e_1,\dots,e_l$ such that $\mathrm{m}(e_1)_{n+1}$ equals zero, and $\mathrm{M}(e_l)$, but not $\mathrm{m}(e_l)$, lies on the line through $e$, and $w\in \bigcap C^*(e_i)$.
  \step 62 For each edge $e$ of the edge path, do:
  \step 73 Compute the roots $c$ of the edge polynomial $p_e(t)$ associated with the edge $e$ of the Newton polytope of $p(\bold{x},\phi+y)$, and append to $L_2$ all pairs $(\phi + c\bold{x}^{-\mathrm{S}(e)},e)$.
  \step 81 Set $L_1 = L_2$.
  \step 90 Replace each pair $(\phi + c\bold{x}^{-\mathrm{S}(e)},e)$ in $L_{\mathrm{finite}}\cup L_1$ by $(\phi + c\bold{x}^{-S(e)}, C)$, where $C$ is the barrier cone of $e$ with respect to $p(\bold{x},\phi+y)$, and return $L_{\mathrm{finite}}\cup L_1$. 
\end{Algorithm}

The input of the Newton-Puiseux algorithm as stated here are a square-free polynomial $p$, an edge $e$ of its Newton polytope, some vector $w$ inducing a total order and an integer $k$, and outputs the first terms of series roots of $p$, sufficiently many to distinguish them from each other on the basis of this output. The order in which the terms of these series are computed is given by the total order induced by $w$. The integer $k$ is a lower bound on the number of terms computed, the actual number is $N$ and can be greater if $k$ is too small to distinguish the series solutions by their first $k$ terms. It is the minimal number not smaller than $k$ for which this is the case. The algorithm proceeds similarly as the classical Newton-Puiseux algorithm: it starts with computing a term $c_1\bold{x}^{\alpha_1}$ for which the highest order terms with respect to $w$ in the expansion of $p(\bold{x}, c_1\bold{x}^\alpha_1 + y)$ cancel. Choosing $\alpha_1 = -\mathrm{S}(e_1)$ for $e_1 = e$, these terms come from those terms of $p$ which are supported on $e_1$. The choice of $c_1$ guarantees that they sum up to zero. See step~1. Having computed the first $N$ terms $c_1\bold{x}^{\alpha_1}, \dots, c_N \bold{x}^{\alpha_N}$, the next term is computed analogously, starting with some edge $e_{N+1}$ of the Newton polytope of $p(c_1\bold{x}^{\alpha_1} + \dots + c_N \bold{x}^{\alpha_N}+y)$, chosen such that the order of the corresponding term is smaller than the order of the previously computed terms. See step~5 and step~7. If an edge polynomial $p_e(t)$ has several pairwise different roots, or if there is more than one possible choice for the edge used in the construction of the next term, an approximate series solution splits into several different approximate solutions. The while-loop which starts in step~2 stops when the number of approximate solutions constructed reaches $|\mathrm{M}(e)_{n+1} - \mathrm{m}(e)_{n+1}|$ and the first $N$ terms of each solution have been computed. It eventually terminates because $p$ is square-free and $\mathbb{K}$ has characteristic zero, see~\cite[Theorem~3.6]{MacDonald}. Clearly, as soon as an edge polynomial $p_{e_N}(t)$ has only simple roots the corresponding solutions constructed can be distinguished by their first $N$ terms. In this case, the support of the corresponding series are contained in a shift of the barrier cone $C(e_N)$, see~\cite[Section~3.10]{MacDonald}.

The following statement is an immediate consequence of the correctness of Algorithm~\ref{alg:NPA} and will be helpful later. See~\cite[Corollary~4.1]{MacDonald} for a similar statement.
\begin{Proposition}\label{prop:algClosure}
Let $p\in\mathbb{K}[\bold{x},y]$ and let $w\in\mathbb{R}^n$ define a total order $\preceq$ on $\mathbb{Q}^n$. Then $\mathbb{K}_{\preceq}((\bold{x}))$ contains $\deg_y (p)$ many series roots of $p$ all of which can be computed by Algorithm~\ref{alg:NPA}.
\end{Proposition}
\begin{proof}
We assume that $p$ is square-free and not a multiple of $y$. If $p$ were not square-free, we could instead consider its square-free part $p/\mathrm{gcd}(p,\frac{\partial}{\partial y}p)$. This does not affect the set of roots but only their multiplicities. We show that there is a unique edge path~$e_1,\dots,e_k$ on the Newton polytope of $p$ such that $w\in \bigcap C^*(e_i)$ and~$\mathrm{m}(e_1)_{n+1} = 0$ and $\mathrm{M}(e_k)_{n+1} = \deg_y (p)$. Algorithm~\ref{alg:NPA} then implies that these edges give rise to $\deg_y (p)$ many pairwise different series roots of $p$ all of which are elements of $\mathbb{K}_{\preceq}((\bold{x}))$: the number of (pairwise different) series roots in $\mathbb{K}_\preceq((\bold{x}))$ coming from an edge $e_i$ is $\mathrm{M}(e_i)_{n+1} - \mathrm{m}(e_i)_{n+1}$, series roots coming from different edges $e_i$ and $e_j$ are different since they can be distinguished by their leading exponents $-\mathrm{S}(e_i)$ and $-\mathrm{S}(e_j)$, and $\mathrm{M}(e_i)_{n+1} - \mathrm{m}(e_i)_{n+1}$ telescopes to $\deg_y(p)$ when summed over all edges of the edge path. To construct such an edge path, let $v_1$ be the vertex of the Newton polytope of $p$ which maximizes  $v\cdot w$ among all vertices $v$ in~$\mathbb{R}^n\times \{0\}$ and let $e_1 = \{v_1,v_2\}$ be the admissible edge whose major vertex $v_2$ maximizes $ v \cdot w$ among all vertices $v$ that are adjacent to $v_1$ by an admissible edge. Then~$w\in C^*(e_1)$. The latter holds because the hyperplane which contains $e_1$ and which is spanned by $v_2-v_1$ and the orthogonal complement of $w$ in $\mathbb{R}^n\times \{0\}$ is a supporting hyperplane for the Newton polytope of $p$. If $v_{2,n+1} = \deg_y (p)$ we are finished as $e_1$ is already the edge path we are looking for. If this is not the case we can extend $e_1$ by an admissible edge $e_2$ whose minor vertex equals the major vertex of $e_1$ and which maximizes $\mathrm{M}(e_2) \cdot w$ among all such edges. Extending the edge path until it cannot be extended any further by an admissible edge results in an edge path with the claimed property. The uniqueness of the path is a consequence of $w$ defining a total order on $\mathbb{R}^n$.
\end{proof}

\section{Finite encodings}\label{sec:}
The Newton-Puiseux algorithm determines the series solutions of a polynomial equation term by term.
The next proposition implies that it can also be used to represent a series by a finite amount of data: its minimal polynomial (or an annihilating square-free polynomial), a total order, and its first few terms with respect to this order.

\begin{Proposition}\label{prop:uniqueness}
Let $p\in\mathbb{K}[\bold{x},y]$ be a square-free and non-constant polynomial, $e$ an admissible edge of its Newton polytope, $w$ an element of $C^*(e)$ inducing a total order $\preceq$ on $\mathbb{Q}^n$ and $a_1\bold{x}^{\alpha_1},\dots,a_N \bold{x}^{\alpha_N}$ the first few terms of a series solution $\phi$ as output by Algorithm~\ref{alg:NPA} when applied to $p$, $e$, $w$ and~$k=0$. Then $\phi$ is the only series solution in $\mathbb{K}_{\preceq}((\bold{x}))$ whose first terms are $a_1\bold{x}^{\alpha_1},\dots,a_N \bold{x}^{\alpha_N}$.
\end{Proposition}
\begin{proof}
The proof of Proposition~\ref{prop:algClosure} shows that $e$ can be extended to an edge path on the Newton polytope of $p$ that gives rise to $\deg_y(p)$ many series solutions in $\mathbb{K}_{\preceq}((\bold{x}))$.
By design of the algorithm, the first $N$ terms of any series solution constructed from $e$ different from $\phi$ differ from $a_1\bold{x}^{\alpha_1},\dots,a_N\bold{x}^{\alpha_N}$. If $e'$ is any other edge of the edge path, then the leading exponent of any series solution resulting from it is $-\mathrm{S}(e')$ and different from the leading exponent $-\mathrm{S}(e)$ of $\phi$, see the proof of~\cite[Lemma 3.7]{MacDonald}.
\end{proof}

It is natural to ask whether there is an upper bound on the number of terms needed to encode an algebraic series. The question is discussed in~\cite{walsh2000polynomial} for series in a single variable over $\mathbb{Q}$. For multivariate series, however, it is still an open problem.
\begin{Problem}
Determine an upper bound on the number of terms needed to encode an algebraic series.
\end{Problem}

We illustrate the Newton-Puiseux algorithm and Proposition~\ref{prop:uniqueness}.

\begin{Example}\label{ex:NPA}
We determine the first terms of a series solution of the equation
\begin{equation*}
p(x,y,z) := 4x^2y+(x^2y+xy^2+xy+y)^2-z^2 = 0
\end{equation*} 
when solved for $z$.
The Newton polytope of $p$ has four admissible edges, one of which is the edge~$e=\{(0,2,0),(0,0,2)\}$. Its barrier cone is $C(e)=\mathrm{cone}\{(1,1),(2,-1)\}$, and $w:=(-\sqrt{2},-1)$ is an element of its dual $C^*(e)$. Its components are linearly independent over $\mathbb{Q}$, therefore it defines a total order $\preceq$ on $\mathbb{Q}^2$. By~\cite[Theorem 3.5]{MacDonald}, and because the projection of $e$ on its last coordinate has length $2$, there are two series solutions~$\phi_1$ and $\phi_2$ of~$p(x,y,z)=0$ in $\mathbb{K}_\preceq((x,y))$. We determine their first terms using Algorithm~\ref{alg:NPA}. The slope of $e$ is~$\mathrm{S}(e) = (0,-1)$, so the solutions have a term of the form $c y$, for some $c\in\mathbb{K}$. The coefficients $c$ are the solutions to $-1+t^2 = 0$. Hence, $y$ is the first term of one series solution, say $\phi_1$, and $-y$ the first term of~$\phi_2$. Furthermore, their support is contained in $(0,1) + \mathrm{cone}\{(1,1),(2,-1)\}$, because $-1+t^2$ has only simply roots. To compute the next term of $\phi_1$, for instance, we consider the polynomial $p(x,y,y+z)$. The edge path on its Newton polytope mentioned in Algorithm~\ref{alg:NPA} consists of the single edge $e=\{(1,2,0),(0,1,1)\}$. Its slope with respect to the last coordinate is~$(-1,-1)$, so its next term is of the form $c xy$, where $c$ is the root of $-2+2t$. By Proposition~\ref{ex:NPA} the series $\phi_1$ and $\phi_2$ can be encoded by $(p,(-\sqrt{2},-1),y)$ and $(p,(-\sqrt{2},-1),-y)$, respectively.
\end{Example}

\section{An effective equality test}\label{sec:equ}

The encoding of an algebraic series by its minimal polynomial, a total order, and its first terms is not unique, and so it is natural to ask if it is possible to decide whether two encodings represent the same series. We clarify this in this section. We begin with explaining how to compare the initial terms of two series when they are given with respect to different total orders. Let~$\phi_1$ and $\phi_2$ be two series solutions of $p(\bold{x},y)=0$ encoded by~$(p,w_1,q_1)$ and $(p,w_2,q_2)$ where $w_1$ and~$w_2$ are elements of $\mathbb{R}^n$ inducing total orders~$\preceq_1$ and~$\preceq_2$ on $\mathbb{Q}^n$ and $q_1=a_1 \bold{x}^{\alpha_1} + \dots + a_N \bold{x}^{\alpha_N}$ and $q_2=b_1 \bold{x}^{\beta_1} + \dots + b_M \bold{x}^{\beta_M}$ are Puiseux polynomials in $\bold{x}$ representing the sum of the first terms of $\phi_1$ and $\phi_2$ with respect to $\preceq_1$ and $\preceq_2$, respectively. In general, we cannot decide whether $\phi_1$ equals $\phi_2$ or not based on $q_1$ and $q_2$ alone. The reason is, there might be a term of $q_2$ which does not appear in $q_1$, because it has not been computed yet. To make sure that this is not the case, we assume that the order of the lowest order term of $q_1$ is less than or equal to the order of the lowest order term of $q_2$ with respect to $w_1$. Let $\sigma = (\sigma_1,\dots,\sigma_N)\in \mathfrak{S}_N$ be a permutation such that $(a_{\sigma_1}\bold{x}^{\alpha_{\sigma_1}},\dots, a_{\sigma_N}\bold{x}^{\alpha_{\sigma_N}})$ is the sequence of terms of $q_1$ when ordered with respect to $w_2$. If $(b_1\bold{x}^{\beta_1},\dots, b_M \bold{x}^{\beta_M})$ does not equal $(a_{\sigma_1}\bold{x}^{\alpha_{\sigma_1}},\dots, a_{\sigma_M}\bold{x}^{\alpha_{\sigma_M}})$, then the series $\phi_1$ and $\phi_2$ cannot be the same. The next example demonstrates that we can also prove the equality of two series by comparing (only finitely many of) their initial terms and estimating their supports.

\begin{Example}\label{ex:nonUniqueness}
The Newton polytope of 
\begin{equation*}
p(x,y,z):= x+y - (1+x+y) z
\end{equation*}
has four admissible edges from each of which we can compute the first terms of a series solution of~$p(x,y,z) = 0$. These series can be encoded by 
\begin{align*}
(p, (-1+1/\sqrt{2},1),1) \quad  &\text{and} \quad (p, (-1+1/\sqrt{2},-1),1), \quad \text{and}\\ 
(p, (-1+1/\sqrt{2},-2),x) \quad &\text{and} \quad (p, (-2+1/\sqrt{2},-1),y).
\end{align*}

We claim that the series~$\phi_1$ represented by $(p, (-1+1/\sqrt{2},-2),x)$ and~$\phi_2$ represented by $(p, (-2+1/\sqrt{2},-1),y)$ are equal. The order of $y$ with respect to~$(-1+1/\sqrt{2},-2)$ is $-2$, and the terms of $\phi_1$ whose order with respect to~$(-1+1/\sqrt{2},-2)$ is at least~$-2$ are
\begin{equation*}
x,-x^2,x^3,-x^4,x^5,-x^6,y.
\end{equation*}
Ordering these terms with respect to $(-2+1/\sqrt{2},-1)$ results in the sequence
\begin{equation*}
y,x,-x^2,x^3,-x^4,x^5,-x^6,
\end{equation*}
whose first term equals the first term of $\phi_2$.
\begin{figure}
\begin{center}
  \begin{tikzpicture}[scale=.2]
    \fill[gray] (-3,-1)--(-3,5.5)--(11,5.5)--(11,-5.5)--(11,-5.5)--(3/2,-11/2)--cycle;
    \draw[->](-6.5,-3)--(12,-3);
    \draw[->](-3,-6.5)--(-3,6.5);
    \draw (-11/2,2.232)--(-3,-1)--(0.481,-11/2);
    \draw (-3,-1) node {$\bullet$};
    \begin{scope}[xshift=25cm]
  \fill[gray] (-3,-1)--(-3,5.5)--(11,5.5)--(11,-3)--cycle;
    \draw[->](-6.5,-3)--(12,-3);
    \draw[->](-3,-6.5)--(-3,6.5);
    \draw (-11/2,-0.6339)--(-3,-1)--(11/2,-2.2448)--(11,-3.05);
    \foreach \x/\y in {-1/-3, 1/-3, 3/-3, 5/-3, 7/-3, 9/-3,
     -3/-1
} \draw (\x,\y) node {$\bullet$};
\end{scope}
    \end{tikzpicture}
\end{center}
    \caption{The leading exponent of the series encoded by $(p, (-2+1/\sqrt{2},-1),y)$, the support of the first terms of $(p, (-1+1/\sqrt{2},-2),x)$ up to the order of $y$ with respect to $(-1+1/\sqrt{2},-2)$, and the (shifted) cones which contain their remaining (non-depicted) supports.} 
    \label{fig:equality}
\end{figure}
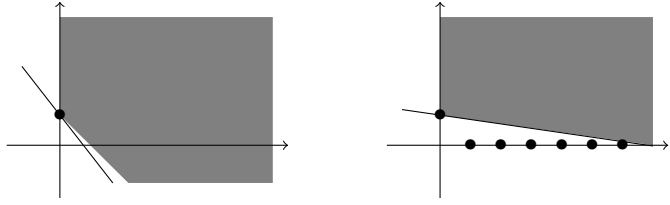
Algorithm~\ref{alg:NPA} shows that the support of $\phi_1$ is contained in a shift of $\mathrm{cone}\{(0,1),(7,-1)\}$ and that the support of $\phi_2$ is contained in a shift of $\mathrm{cone}\{(0,1),(1,-1)\}$. These cones are compatible in the sense that their sum is a strictly convex cone. Consequently, there is a total order $\preceq$ on $\mathbb{Q}^2$ that is compatible with both of them, and so~$\mathbb{K}_{\preceq}((x,y))$ contains both $\phi_1$ and $\phi_2$. Since, by construction, $\phi_1$ and $\phi_2$ are roots of $p$, and since $p$ has degree~$1$ with respect to $z$ and $\mathbb{K}_{\preceq}((x,y))$ is a field and therefore can contain at most one solution of $p(x,y,z) = 0$, the series $\phi_1$ and $\phi_2$ have to be the same. The series represented by the other encodings turn out to be different from $\phi_1$ and different from each other. See Figure~\ref{fig:equality}.
\end{Example}

The reasoning in Example~\ref{ex:nonUniqueness} relied on the cones output by Algorithm~\ref{alg:NPA} not being too big. Although they are not always minimal, we believe that they are always in certain situations. We will have more to say about this at the end of this section and in Section~\ref{sub:support}. For the moment we just state the following conjecture.
\begin{Definition}
A polynomial $p\in\mathbb{K}[x][y]$ is said to be content-free if the greatest common divisor of its coefficients in $\mathbb{K}[x]$ is $1$.
\end{Definition}
\begin{Conjecture}\label{conj:minimality}
Let $a_1\bold{x}^{\alpha_1}+\dots+a_N \bold{x}^{\alpha_N}$ be the sum of the first terms of a series solution $\phi$ of $p(\bold{x},\phi) = 0$ output by Algorithm~\ref{alg:NPA}, and let $C$ be the cone that goes with it. If
$p\in\mathbb{K}[\bold{x}][y]$ is content-free and if $k$, and hence $N$, is sufficiently large, then $C$ is minimal in the sense that there is no smaller cone satisfying
\begin{equation*}
\mathrm{supp}(\phi) \subseteq \{\alpha_1,\dots,\alpha_{N-1}\} \cup (\alpha_N + C).
\end{equation*}
\end{Conjecture}

We now prove the main theorem of this section.

\begin{Theorem}\label{thm:equality}
The equality of two algebraic series can be decided effectively.
\end{Theorem}
\begin{proof}
Let~$\phi_1$ and $\phi_2$ be two series solutions of $p(\bold{x},y)=0$ encoded by~$(p,w_1,q_1)$ and $(p,w_2,q_2)$ as above. Assume that the order of the lowest order term of $q_1$ is less than or equal to the order of the lowest order term of $q_2$ with respect to $w_1$, and assume that the terms of $q_2$ equal the first few terms of $q_1$ when ordered with respect to $w_2$. If this were not the case, then $\phi_1 \neq \phi_2$ and there is nothing left to show. Furthermore, assume that the lowest order term $b_M \bold{x}^\beta_M$ of $q_2$ with respect to $w_2$ equals the lowest order term $a_N \bold{x}^\alpha_N$ of $q_1$ with respect to $w_1$. Otherwise extend $q_2$ by additional terms of $\phi_2$ until it involves a term which does not appear in $q_1$ and extend $q_1$ by additional terms of $\phi_1$ until again all terms of $q_2$ appear in $q_1$. Let $C_1$ be the cone output by Algorithm~\ref{alg:NPA} such that $\mathrm{supp}(\phi_1)\subseteq\{\alpha_1,\dots,\alpha_{N-1}\}\cup \left( \alpha_N + C_1 \right)$, and let $C_2$ be the corresponding cone such that $\mathrm{supp}(\phi_2)\subseteq\{\beta_1,\dots,\beta_{M-1}\}\cup \left( \beta_M + C_2 \right)$. There are now two situations: either $C_1$ is a subset of $C_2$, or it is not. If $C_1\subseteq C_2$, we have $w_2\in C_1^*$, hence $\phi_1\in\mathbb{K}_{\preceq_2}((\bold{x}))$, and it follows that $\phi_1=\phi_2$ as in Example~\ref{ex:nonUniqueness}: the terms of $q_2$ being equal to the first terms of $q_1$ when ordered with respect to $w_2$ implies that $\phi_1$ needs to be different from the series roots of $p$ in $\mathbb{K}_{\preceq_2}((\bold{x}))$ different from $\phi_2$. However, it cannot be different from $\phi_2$ as $\mathbb{K}_{\preceq_2}((\bold{x}))$ cannot contain more than $\deg_y(p)$ many roots of $p$. If we assume that $C_1$ is the minimal cone for which~$\mathrm{supp}(\phi_1)\subseteq\{\alpha_1,\dots,\alpha_{N-1}\}\cup \left( \alpha_N + C_1 \right)$, then we can also conclude that~$\phi_1\neq \phi_2$, if $C_1\nsubseteq C_2$.

Since we do not know how to determine an $N$ for which the cone output by Algorithm~\ref{alg:NPA} is minimal (if such an $N$ exists at all), we have to argue differently. We explain how to prove (or disprove) that $\mathrm{supp}(\phi_1) \subseteq \{\alpha_1,\dots, \alpha_{N-1}\} \cup \left( \alpha_N + C_2 \right)$ using basic properties of D-finite series. To simplify the argument we explain how to decide whether the support of an algebraic series $\phi$ is contained in a given rational cone $C$. Wlog assume that $\mathrm{supp}(\phi)\subseteq \mathbb{Z}^n$. 
Let $\omega=(\omega_1,\dots,\omega_n)\in\mathbb{Z}^n$ be such that for each $i\in\mathbb{Z}$ there are only finitely many $\alpha\in\mathrm{supp}(\phi)$ for which $\alpha\cdot \omega = i$ and none if $i< i_0$ for some $i_0\in\mathbb{Z}$, and consider the series 
\begin{equation*}
\tilde{\phi} (\bold{x},t) := \phi (x_1 t^{\omega_1},\dots, x_n t^{\omega_n} ) 
\end{equation*}
Since $\phi$ is algebraic, so is~$\tilde{\phi}$, and because every algebraic series is D-finite, $\tilde{\phi}$ is D-finite too. 
Let
\begin{equation*}
[\phi]_C(\bold{x}) := \sum_{I\in C\cap \mathbb{Z}^n} \left([\bold{x}^I] \phi\right) \bold{x}^I
\end{equation*}
be the restriction of $\phi$ to $C$, and let $[\tilde{\phi}]_C(\bold{x},t) := [\phi]_C(x_1t^{\omega_1},\dots,x_n t^{\omega_n})$ denote the restriction of~$\tilde{\phi}$ to $C\times \mathbb{R}_{\geq 0}$.
By closure properties of D-finite functions, the series~$\tilde{\phi} - [\tilde{\phi}]_C$ is D-finite. In particular, seen as a series in $t$, the coefficients of $\tilde{\phi} - [\tilde{\phi}]_C$ satisfy a linear recurrence relation of the form 
\begin{equation}\label{eq:recurrence}
q_0(k)c_k + q_1(k)c_{k+1} + \dots + q_r(k)c_{k+r} = 0
\end{equation}
with $q_0,\dots,q_r \in \mathbb{K}[\bold{x}][k]$. Whether $\mathrm{supp}(\phi) \subseteq C$ can therefore be verified by checking if $\tilde{\phi} - [\tilde{\phi}]_C = 0$, which can be done by comparing finitely many of its initial terms to zero.
\end{proof}

The equality test for algebraic series is effective because the closure properties of D-finite functions it is based on can be performed effectively. However, as was explained in~\cite{bostan2017hypergeometric}, it can be computationally quite expensive to do so. The explicit computation of the recurrence relation~\eqref{eq:recurrence} could be avoided by deriving an upper bound for its order and the largest integer root of its leading coefficient polynomial and checking the equation which is to be verified for sufficiently many initial terms. Yet, these bounds may be so high that this is not practicable either. It is therefore preferable to have an equality test based on the correctness of Conjecture~\ref{conj:minimality}.

We end this section with an example that illustrates Conjecture~\ref{conj:minimality}.
\begin{Example}
Let
\begin{equation*}
p(x,y,z) := a_{000} + a_{010}y + z\left( a_{001} + a_{011}y +  a_{101}x + a_{111}xy\right)
\end{equation*}
be a polynomial in $x,y$ and $z$, and assume that its coefficients $a_{ijk}$ are undetermined and non-zero. We apply Algorithm~\ref{alg:NPA} to $p$, the edge $e = \{(0,0,0),(0,0,1)\}$, and any $w\in C^*(e)$ that induces a total order on $\mathbb{Q}^2$ and construct the first terms of a series solution $\phi$ of $p(x,y,\phi) = 0$. The slope of $e$ with respect to the last coordinate is $-\mathrm{S}(e) = (0,0)$, its edge polynomial is $p_e(t) = a_{000} + t a_{001}$, and so the first term of $\phi$ is 
\begin{equation*}
\phi_1 = -\frac{a_{000}}{a_{001}}.
\end{equation*}
To construct the next term of $\phi$ we consider 
\begin{align*}
p(x,y,z+\phi_1) = &y\left( a_{010} - \frac{a_{000}a_{011}}{a_{001}} \right) - x \left( \frac{a_{000}a_{101}}{a_{001}} + y\frac{a_{000}a_{111}}{a_{001}} \right)\\ 
&+ z\left( a_{001} + a_{011}y +  a_{101}x + a_{111}xy\right).
\end{align*}
Before we do so, we note that the algorithm implies that 
\begin{equation*}
\mathrm{supp}(\phi) \subseteq \mathrm{cone}\{(1,0),(0,1)\}.
\end{equation*}
The condition for no term other than $\phi_1$ having its exponent on $\mathbb{R}_{\geq 0}\cdot (0,1)$ is that the coefficient of $y$ in $p(x,y,z+\phi_1)$ is zero, i.e
\begin{equation*}
 a_{010} - \frac{a_{000}a_{011}}{a_{001}} = 0,
\end{equation*}
or equivalently
\begin{equation}\label{eq:Cancel1}
\frac{a_{010}}{a_{000}} = \frac{a_{011}}{a_{001}}.
\end{equation}
We assume that equation~\eqref{eq:Cancel1} is satisfied. Then $e_2 = \{(1,0,0),(0,0,1)\}$ is the edge of the Newton polytope of $p(x,y,z+\phi_1)$ that is used to construct the next term of $\phi$. Its slope is $-\mathrm{S}(e_2) = (1,0)$, and the corresponding edge polynomial is $p_{e_2}(t) =  -\frac{a_{000}a_{101}}{a_{001}} +ta_{001}$, and so 
\begin{equation*}
\phi_2 = \frac{a_{000}a_{101}}{a_{001}^2}x.
\end{equation*}
Again, the algorithm predicts that
\begin{equation*}
\mathrm{supp}(\phi) \subseteq \{(0,0)\} \cup ( (1,0) + \mathrm{cone}\{(1,0), (0,1)\}),
\end{equation*}
and again, the condition for no term other than $\phi_2$ having its exponent on $(1,0) + \mathbb{R}_{\geq 0}\cdot (0,1)$ is that the coefficient of $xy$ in 
\begin{align*}
p_2(x,y,z+\phi_1 + \phi_2) = xy \frac{ a_{000}\left( a_{011}a_{101} - a_{001}a_{111}\right)}{a_{001}^2} +\\ 
x^2 \left( \frac{a_{000}a_{101}^2}{a_{001}^2} + y \frac{a_{000}a_{101}a_{111}}{a_{001}^2} \right)+ 
z\left( a_{001} + a_{011}y +  a_{101}x + a_{111}xy\right)
\end{align*}
equals zero.
Assuming that $a_{000}$ is different from zero, this is the case if and only if 
\begin{equation}\label{eq:Cancel2}
\frac{a_{011}}{a_{001}} = \frac{a_{111}}{a_{101}}.
\end{equation}
We assume that in addition to equation~\eqref{eq:Cancel1} also equation~\eqref{eq:Cancel2} holds and define
\begin{equation*}
\lambda := \frac{a_{010}}{a_{000}} = \frac{a_{011}}{a_{001}} = \frac{a_{111}}{a_{101}}.
\end{equation*}
Then 
\begin{align*}
p(x,y,z) &= a_{000} + a_{010}y + z\left( a_{001} + a_{011}y +  a_{101}x + a_{111}xy\right)\\
   &= a_{000} + \lambda a_{000}y + z\left( a_{001} + \lambda a_{001}y +  a_{101}x + \lambda a_{101}xy\right)\\
   &= a_{000} (1+\lambda y) + z\left( a_{001} (1+ \lambda y) +  a_{101}x (1 + \lambda y)\right)\\
   &= (1+\lambda y) \left( a_{000} + z\left( a_{001} +  a_{101}x \right)\right),
\end{align*}
so $p$ is not content-free.
\end{Example}

\section{The support}\label{sub:support}

In Example~\ref{ex:nonUniqueness} we saw that the number of admissible edges of the Newton polytope of a polynomial equation is not necessarily bounded by the number of series solutions the equation has. It can happen that different edges give rise to the same series solution. For the purpose of encoding a series, any edge is as good as any other edge as long as it gives rise to the same series. However, when interested in information about the support of a series solution, it is important to know all the edges which give rise to this series solution. 

\begin{Example}\label{ex:cones}\label{ex:support}
In Example~\ref{ex:nonUniqueness} we saw that the Newton polytope of 
\begin{equation*}
p(x,y,z) = x+y - (1+x+y) z
\end{equation*}
has two admissible edges,
\begin{equation*}
e_1 = \{(0,0,1),(1,0,0)\} \quad \text{and} \quad e_2 = \{(0,0,1),(0,1,0)\},
\end{equation*}
that give rise to two encodings, 
\begin{equation*}
(p, (-1+1/\sqrt{2},-2),x) \quad \text{and} \quad (p, (-2+1/\sqrt{2},-1),y), 
\end{equation*}
of one and the same series solution $\phi$ of $p(x,y,z) = 0$. The barrier cones of these edges are 
\begin{equation*}
C(e_1) = \mathrm{cone}\{(1,0),(-1,1)\} \quad \text{and} \quad  C(e_2) = \mathrm{cone}\{(0,1),(1,-1)\},
\end{equation*}
so that, by Algorithm~\ref{alg:NPA}, we have 
\begin{equation*}
\mathrm{supp}(\phi) \subseteq (1,0) + C(e_1) \quad \text{as well as} \quad \mathrm{supp}(\phi) \subseteq (0,1) + C(e_2).
\end{equation*}

\end{Example}

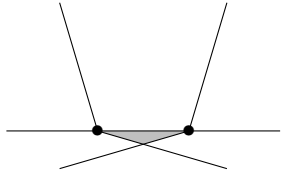
\begin{figure}
\begin{center}
  \begin{tikzpicture}[scale=.2]
     \fill[lightgray] (-3,-3)--(3,-3)--(0,-3.88235)--cycle; 
     \draw [black] (-9,-3)--(9,-3);
     \draw (-3,-3)--(-5.5,5.5);
     \draw (-3,-3)--(5.5,-5.5);
     \draw [black] (-3,-3) node {$\bullet$};
     \draw (3,-3)--(5.5,5.5);
     \draw (3,-3)--(-5.5,-5.5);
     \draw [black] (3,-3) node {$\bullet$};
  \end{tikzpicture}
\end{center}
 \caption{The (only) two vertices $v_1, v_2$ of a (not depicted) convex set $P\subseteq\mathbb{R}^2$, and two strictly-convex cones $C_1, C_2$ attached to them such that $P\subseteq (v_1+ C_1)\cap (v_2 + C_2)$. The complement of the line going through $v_1,v_2$ is the union of two open half-planes. The lower half-plane has bounded intersection with $(v_1 + C_1)\cap(v_2 + C_2)$ and does not contain any vertices of $P$. So $\mathrm{conv}(\{v_1,v_2\}$ is a bounded face of $P$.}
 \label{fig:boundedFace}
\end{figure}

The following proposition and corollary make this more precise. They show that the convex hull of the support of an algebraic series is a polyhedral set and indicate how its bounded faces can be computed.

\begin{Proposition}\label{prop:support}
For every series root $\phi$ of a non-zero square-free polynomial~$p\in\mathbb{K}[\bold{x},y]$, there is a surjection from the set of edges of the Newton polytope of $p$ which give rise to $\phi$ to the set of vertices of the convex hull of its support. 
\end{Proposition}
\begin{proof}
We claim that the function that maps an edge $e$ to the negative of its slope $-\mathrm{S}(e)$ has the required properties. If~$e$ is an edge that gives rise to the series solution $\phi$, then $-\mathrm{S}(e)$ is necessarily a vertex of $\mathrm{Newt}(\varphi)$, the convex hull of its support, as it is the maximal element with respect to a total order induced by some irrational vector. Therefore, the function is well-defined. The function is also surjective, because for every vertex $\alpha$ of $\mathrm{Newt}(\varphi)$ there is some $w\in\mathbb{R}^n$ that induces a total order~$\preceq$ on~$\mathbb{Q}^n$ with respect to which $\alpha$ is the maximal element of $\mathrm{Newt}(\phi)$. In particular, $\phi$ is an element of $\mathbb{K}_{\preceq}((\bold{x}))$. Since all series solutions of $p(\bold{x},y) = 0$ in~$\mathbb{K}_{\preceq}((\bold{x}))$ can be constructed by the Newton-Puiseux algorithm, see Proposition~\ref{prop:algClosure}, there is also an edge $e$ that gives rise to $\phi$. 
\end{proof}

\begin{Corollary}
The convex hull of the support of an algebraic series is polyhedral.
\end{Corollary}
\begin{proof}
The set of edges of $\mathrm{Newt}(p)$ is finite, hence so is the set of vertices of $\mathrm{Newt}(\varphi)$. Any bounded face of $\mathrm{Newt}(\varphi)$ is a bounded face of the convex hull of $\mathrm{vert}(\mathrm{Newt}(\varphi))$. It follows that the set of bounded faces of $\mathrm{Newt}(\varphi)$ is finite too. There is at most one bounded face of the convex hull of $\mathrm{vert}(\mathrm{Newt}(\varphi))$ that is not a bounded face of $\mathrm{Newt}(\varphi)$. So the set of unbounded faces is finite too. 
\end{proof}

The proof of Proposition~\ref{prop:support} indicates how the bounded faces of $\mathrm{Newt}(\varphi)$ can be computed. Identify the admissible edges of $\mathrm{Newt}(p)$ that give rise to $\varphi$. The negative of their slopes are the vertices of $\mathrm{Newt}(\varphi)$. Compute the terms of $\varphi$ supported in the convex hull of $\mathrm{vert}(\mathrm{Newton}(\varphi))$. If they make up all of $\varphi$, then $\varphi$ is a Puiseux polynomial, and $\mathrm{Newt}(\varphi)$ equals the convex hull of its vertices. If there is a term that is not  
supported on the convex hull of $\mathrm{vert}(\mathrm{Newt}(\varphi))$, then consider any line segment connecting its exponent with a point in the interior of the convex hull of $\mathrm{vert}(\mathrm{Newt}(\varphi))$. It intersects a face of the convex hull of $\mathrm{vert}(\mathrm{Newton}(\varphi))$. It is the only face that is not a bounded face of $\mathrm{Newt}(\varphi)$. 

The identification of the bounded faces of $\mathrm{Newt}(\varphi)$ of codimension $<n$ relies on the computation of terms of $\varphi$. The computation of (some of) them can sometimes be avoided, given that for any vertex~$v$ of $\mathrm{Newt}(\varphi)$ one can compute a strictly convex cone $C_v$ such that $\mathrm{supp}(\phi)\subseteq v + C_v$. Recall that for each edge $e$ that gives rise to $\phi$ and for every~$w\in C^*(e)$ that defines a total order, Algorithm~\ref{alg:NPA} provides its first $N$ terms $a_1\bold{x}^{\alpha_1},\dots,a_N\bold{x}^{\alpha_N}$ with respect to $w$ and a strictly convex cone $C$ compatible with $w$ such that $\mathrm{supp}(\phi)\subseteq \{\alpha_1,\dots,\alpha_{N-1}\}\cup (\alpha_N + C)$. The cone $C_{\alpha_1}$ generated by $C$ and~$\{\alpha_2-\alpha_1,\dots,\alpha_N-\alpha_1\}$ has the property that $\mathrm{supp}(\phi)\subseteq \alpha_1 + C_{\alpha_1}$. It is also strictly convex since it is compatible with~$w$.

\begin{Example}\label{ex:supportCont}
We continue with Example~\ref{ex:support}. Proposition~\ref{prop:support} implies that the vertices of the convex hull of the support of $\phi$ are $v_1 = (1,0)$ and $v_2 = (0,1)$. Apart from the vertices $v_1$ and $v_2$ itself, the only possible bounded face is the convex hull of $v_1$ and $v_2$. Since
$\mathrm{supp}(\phi)\subseteq v_1 + \mathrm{cone}\{(1,0), (-1,1)\}$ and $\mathrm{supp}(\phi) \subseteq v_2 + \mathrm{cone}\{(0,1), (1,-1)\}$,  
the line through $v_1$ and $v_2$ supports $\mathrm{conv}(\mathrm{supp}(\phi))$, hence $\mathrm{conv}(\{v_1,v_2\})$ is a bounded face. See also Figure~\ref{fig:boundedFace} for an illustration of a slightly different example.
\end{Example}

The observation made in Example~\ref{ex:supportCont} and illustrated in Figure~\ref{fig:boundedFace} motivates the following proposition.
\begin{Proposition}\label{prop:face}
Let $V$ be a subset of $\mathrm{vert}(\mathrm{Newt}(\phi))$, and for each $v\in V$, let $C_v$ be a strictly convex cone such that $\mathrm{supp}(\phi)\subseteq v+ C_v$. Assume that there is a hyperplane which contains $V$ but no other vertices of $\mathrm{Newt}(\phi)$. Furthermore, assume that its complement is the union of two half-spaces one of which has bounded intersection with $\bigcap_{v\in V} (v+C_v)$ and does not contain any vertices of $\mathrm{Newt}(\phi)$. Then $\mathrm{conv}(V)$ is a bounded face of $\mathrm{Newt}(\phi)$.
\end{Proposition}
\begin{proof}
If $\mathrm{conv}(V)$ were not a (bounded) face of $\mathrm{Newt}(\phi)$, neither of the two (open) half-spaces would have a trivial intersection with $\mathrm{Newt}(\phi)$. In particular, the bounded intersection of one of them with $\bigcap_{v\in V} (v+C_v)$ would be non-empty, and hence contain a vertex of $\mathrm{Newt}(\phi)$. A contradiction to the assumptions, so the hyperplane supports $\mathrm{Newt}(\phi)$, and $\mathrm{conv}(V)$ is a bounded face. 
\end{proof}

It can happen that for each hyperplane that contains $V$ but no other vertices of $\mathrm{supp}(\varphi)$ the half-spaces that make up its complement have unbounded intersection with $\bigcap_{v\in V} (v+C_v)$ or contain vertices of $\mathrm{conv}(\mathrm{supp}(\phi))$ different from $V$. 

\begin{Example}
Consider $P = \mathrm{conv}\{(-1,0), (1,0), (0,-1) \}$, and the cones $C_{(-1,0)} = \mathrm{cone}\{(-1,2), (1,-1)\}$ and $C_{(1,0)} = \mathrm{cone}\{(1,2), (-1,-1)\}$. The line $\mathbb{R}\cdot (1,0)$ is tangent to $P$ at $\mathrm{conv}\{(-1,0), (1,0)\}$. However, the corresponding upper half-plane has unbounded intersection with the intersection of $(-1,0) + C_{(-1,0)}$ and $(1,0) + C_{(1,0)}$ while the lower half-plane contains the vertex $(0,-1)$.
\end{Example}

We pointed out that for deciding whether two algebraic series are equal or not, it is convenient that the cones output by Algorithm~\ref{alg:NPA} are not too big. The problem of computing the unbounded faces of $\mathrm{Newt}(\varphi)$ raises the question whether they are even minimal. The following example shows that this does not need to be the case.
\begin{Example}
Let $\mathbb{K}_{\preceq}((\bold{x}))$ be the field of Puiseux series induced by $w = (-\sqrt{2},-1)$. The series solution of
 \begin{equation*}
  p(x,y,z) = (1-x)((1-y)z-1) = 0, 
 \end{equation*}
therein is the geometric series 
\begin{equation*}
 \phi = 1 + y + y^2 + \dots  
\end{equation*}
Though the convex hull of its support is the cone generated by $(1,0)$, Algorithm~\ref{alg:NPA} only shows that 
\begin{equation*}
 \mathrm{supp}(\phi) \subseteq \mathrm{cone}\{(1,0), (0,1)\}.
\end{equation*}
The difference between the two cones is caused by the polynomial~$p\in\mathbb{K}[x,y][z]$ not being content-free: first getting rid of its content, and then applying Algorithm~\ref{alg:NPA} results in a cone that is minimal. 
\end{Example}
The non-primitivity is not the only possible reason for a cone output by Algorithm~\ref{alg:NPA} not being minimal.

\begin{Example}\label{ex:minimal0}
One of the solutions of 
\begin{equation*}
 p(x,y,z )= 1+x+y + (1 + xy +2y)z +yz^2 = 0
\end{equation*}
is
\begin{equation*}
 \frac{-1-2y-xy + \sqrt{1-2xy+4xy^2+x^2y^2}}{2y}.
\end{equation*}
Let $\varphi$ be its series expansion whose first terms with respect to $w = (-1+1/\sqrt{2},-1)$ are
\begin{equation*}
-1-x+xy+x^2y^2 + \dots. 
\end{equation*}
The closed form of $\phi$ together with Newton's binomial theorem implies that the minimal cone containing $\mathrm{supp}(\phi)$ is $\mathrm{cone}\{(1,0), (1,2)\}$, though Algorithm~\ref{alg:NPA} only shows that it is contained in $\mathrm{cone}\{(1,0), (0,1)\}$. However, the algorithm also shows that
\begin{equation*}
 \mathrm{supp}(\phi) \subseteq \{(0,0)\} \cup ( (1,0) + \mathrm{cone}\{(1,1), (0,1)\}),
\end{equation*}
where now $\mathrm{cone}\{(1,1), (0,1)\}$ is the minimal cone having this property, and computing another term and another cone, we find that
\begin{equation*}
 \mathrm{supp}(\phi) \subseteq \{(0,0), (1,0)\} \cup ((1,1) + \mathrm{cone}\{(1,1), (1,2)\}), 
\end{equation*}
where the cone $\mathrm{cone}\{(1,1), (1,2)\}$ is not only minimal but also has the property that its sides~$(1,1)+\mathbb{R}_{\geq 0}\cdot (1,1)$ and $(1,1) + \mathbb{R}_{\geq 0}\cdot (1,2)$ contain infinitely many elements of~$\mathrm{supp}(\phi)$.
\begin{figure}
\begin{center}
  \begin{tikzpicture}[scale=.2]
   \fill[gray!45] (-3,-3)--(5.5,-3)--(5.5,5.5)--(-3,5.5)--cycle;
   \fill[gray!75] (-1,-3)--(5.5,-3)--(5.5,5.5)--(-1,5.5)--cycle;
    \fill[black!65] (-1,-1)--(5.5,5.5)--(9/4,5.5)--cycle;
    \draw (11/2,-5.4896)--(-3,-3)--(-11/2,-2.268);
    \draw[->](-6.5,-3)--(6.5,-3);
    \draw[->](-3,-6.5)--(-3,6.5);
    \foreach \x/\y in {
    -3/-3, -1/-3, -1/-1, 1/1, 1/3
   } \draw (\x,\y) node {$\bullet$};
  \end{tikzpicture}
\end{center}
 \caption{The first terms of the series $\phi$ from Example~\ref{ex:minimal1} and the (shifted) cones, as output by the Newton-Puiseux algorithm, that contain the support of the subsequent terms.}
 \label{fig:minimalCone}
\end{figure}
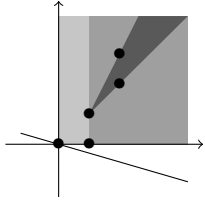

\end{Example}

For the series in the previous examples we could easily decide whether the corresponding cones given by the Newton-Puiseux algorithm were minimal or not, because the series were algebraic of degree~$1$ and~$2$, respectively, and therefore very explicit. It remains to clarify how this can be decided in general. Let $a_1 \bold{x}^{\alpha_1},\dots, a_N \bold{x}^{\alpha_N}$ be the first terms of $\phi$, and let~$C$ be the cone output by Algorithm~\ref{alg:NPA} such that $\mathrm{supp}(\phi) \subseteq \{\alpha_1,\dots,\alpha_{N-1}\}\cup (\alpha_N+C)$. If for each edge of $\alpha_N+C$ the exponent $\alpha_N$ is not the only element of $\mathrm{supp}(\phi)$ it contains, then $C$ is minimal and its minimality can be verified simply by determining some of these elements. 

\begin{Example}\label{ex:minimal1}
We continue Example~\ref{ex:minimal0}. The cone $\mathrm{cone}\{(1,1), (1,2)\}$ for which $\mathrm{supp}(\phi)\subseteq \{(0,0), (1,0)\}\cup ((1,1)+ \mathrm{cone}\{(1,1), (1,2)\})$ is minimal because the first terms of $\phi$ with respect to $w = (-1+1/\sqrt{2},-1)$ are
\begin{equation*}
-1 - x + x y + x^2 y^2 - x^2 y^3 +\dots 
\end{equation*}
and $(2,2)\in(1,1)+\mathbb{R}_{>0} \cdot (1,1)$ and $(2,3) \in(1,1) + \mathbb{R}_{>0} \cdot (1,2)$. See Figure~\ref{fig:minimalCone}.
\end{Example}

Though this is obviously a sufficient condition, it is not a necessary one. It can happen that the cone $C$ output by Algorithm~\ref{alg:NPA} is minimal but $\alpha_N$ is the only element of $\mathrm{supp}(\phi)$ that lies on an edge of $\alpha_N + C$. 

\begin{Example}
Given the polynomial $p(x,y,z) = (1-x)(z-y)-1$, the edge $e=\{(0,0,1), (0,1,0)\}$ of its Newton polytope, the vector $w = (-1,\sqrt{2})$ and $k=0$, Algorithm~\ref{alg:NPA} returns $y$, the first term of the series root $\phi = y + \sum_{n=0}^\infty x^n$, and $C = \mathrm{cone}\{(1,0),(0,-1)\}$. Though $C$ is minimal, the only element of the support of $\phi$ that lies on $(0,1) + \mathbb{R}\cdot (1,0)$ is $(0,1)$.
\end{Example}

Proving the minimality of a cone output by Algorithm~\ref{alg:NPA} in such situations relates to the following open problem. 

\begin{Problem}\label{prob:faces}
Given an algebraic series in terms of its minimal polynomial, a total order and its first terms with respect to this order, determine the unbounded faces of the convex hull of its support.
\end{Problem}

We refer to~\cite[Theorem~1.5]{aroca2022minimal} for further (yet not effective) results on the structure of the convex hull of the support of series algebraic over a field of Laurent series. 

\section{Rational solutions}

We explain how to determine the rational solutions of a polynomial equation. It turns out to be easy to verify whether a series is polynomial once a bound on the set of vertices of the convex hull of its support is known. 
\begin{Proposition}
A series $\varphi$ is a (Puiseux) polynomial if and only if $\mathrm{supp}(\varphi)$ is a subset of the convex hull of $\mathrm{vert}(\mathrm{Newt}(\varphi))$.
\end{Proposition}

The numerator of any rational solution of a polynomial equation divides the latter's leading coefficient. By multiplying with the leading coefficient, any rational solution is transformed into a polynomial solution of an associated polynomial equation. We have thereby proven the following.

\begin{Corollary}
A series $\varphi$ is a rational solution of $p = 0$ if and only if $\mathrm{lc}_y(p) \varphi$ is a polynomial solution of $p(y/\mathrm{lc}_y(p)) = 0$.
\end{Corollary}

We refer to~\cite{alonso1992computational} and Proposition~2.4 and~Corollary~2.5 therein for how to identify polynomial and rational series solutions based on the implicit function theorem.
    
\section{The number of series solutions}

Part of the input of the Newton-Puiseux algorithm as specified in Section~\ref{sec:prelim} are a polynomial, an admissible edge $e$ of its Newton polytope and a vector $w$ of the dual of its barrier cone inducing a total oder on $\mathbb{Q}^n$. It is natural to ask how the series constructed from $e$ depend on $w$, that is, whether $w$ only affects the order in which the terms of the series are constructed. It turns out that in certain conditions the answer is affirmative.

\begin{Proposition}\label{prop:order}
Let $e$ be an edge of the Newton polytope of a square-free polynomial $p\in\mathbb{K}[\bold{x},y]$ whose associated edge polynomial $p_{e}(t)$ is square-free too. Then the set of series roots of $p$ which result from $e$ and an element $w\in C^*(e)$ that induces a total order on $\mathbb{Q}^n$ does not depend on the choice of $w$.
\end{Proposition}
\begin{proof}
Let $w_1, w_2$ be two elements of $C^*(e)$ whose components are linearly independent over $\mathbb{Q}$, and let $\leq_1$ and $\leq_2$ be the corresponding total orders on $\mathbb{Q}^n$. The field $\mathbb{K}_{\leq_1}((\bold{x}))$ contains $\deg_y(p)$ many series roots of which precisely the $\mathrm{M}(e)_{n+1}-\mathrm{m}(e)_{n+1}$ many series resulting from $e$ and $w_1$ have the leading exponent $-\mathrm{S}(e)$. The edge polynomial $p_e(t)$ associated with $e$ is square-free, so the supports of these series are contained in the barrier cone $C(e)$ of $e$. The barrier cone is independent of $w_1$. Hence, the supports of the series roots constructed from $e$ and $w_2$ are contained in $C(e)$. So these series are elements of $\mathbb{K}_{\leq_1}((\bold{x}))$ too. Since $\mathbb{K}_{\leq_1}((\bold{x}))$ cannot contain more than $\deg_y(p)$ many series solutions and because the leading exponents of the series roots constructed from $e$ and $w_2$ are $-\mathrm{S}(e)$, the series roots constructed from $e$ and $w_1$ and $e$ and $w_2$ are the same.
\end{proof}

The following example\footnote{This example was pointed out by one of the reviewers to whom we express our sincere thanks.} shows that the assumption of the Proposition cannot be weakened: if the edge polynomial associated with the edge $e$ of the Newton polytope of a polynomial is not square-free, then the set of series constructed from $e$ may depend on $w\in C^*(e)$.

\begin{Example}\label{ex:complicated}
The series roots of $p(x,y,z) := 1+x+y+2z+z^2$ are
\begin{equation*}
\phi_{1,2} = -1 \pm iy^{1/2} \sum_{k=0}^\infty \binom{1/2}{k}x^ky^{-k} \quad \text{and} \quad  \phi_{3,4} = -1 \pm ix^{1/2} \sum_{k=0}^\infty \binom{1/2}{k}x^{-k}y^k.
\end{equation*}
The Newton polytope of $p$ has three admissible edges one of which is $e = \{(0,0,2), (0,0,0)\}$. Its barrier cone is $C(e) = \mathrm{cone}\{(1,0), (0,1)\}$, hence $C^*(e) = \mathrm{cone}\{(-1,0), (0,-1)\}$. The supports of the series roots of $p$ imply that $C^*(e)$ decomposes into the union of $C_1 = \mathrm{cone}\{(-1,-1), (0,-1)\}$ and $C_2 = \mathrm{cone}\{(-1,0), (-1,-1)\}$ such that the following holds: if $w\in C_1$, then the series constructed from $e$ and $w$ are $\phi_1$ and $\phi_2$. However, if $w\in C_2$, then these series are $\phi_3$ and $\phi_4$. 
\end{Example}

Given an edge $e$ of $\mathrm{Newt}(p)$, Example~\ref{ex:complicated} shows that to see whether $-\mathrm{S}(e)$ is a vertex of $\mathrm{Newt}(\varphi)$ it is not enough to pick any $w\in C^*(e)$ that induces a total order on $\mathbb{Q}^n$ and to verify if $e$ and $w$ give rise to $\phi$. The series constructed from $e$ may depend on $w$, and a priori it is not clear how $w$ should be chosen.  
So how could the set of vertices of $\mathrm{Newt}(\varphi)$ be computed instead? The set of vertices is a subset of 
$S :=\left\{ -\mathrm{S}(e) : e \text{ an admissible edge of } \mathrm{Newt}(p) \right\}$. By computing the terms of $\varphi$ that are supported in $\mathrm{conv}(S)$ and estimates of the support of its remaining terms, one might be able to identify the vertices of $\mathrm{Newt}(\varphi)$. 

Proposition~\ref{prop:order} implies that under the condition that the edge polynomials of the admissible edges of the Newton polytope of a polynomial are square-free, the polynomial has only finitely many series roots. We believe that the latter is true in general. Despite of Example~\ref{ex:complicated}, we also believe that the following conjecture holds.

\begin{Conjecture}\label{conj:cardinality}
For each edge $e$ of the Newton polytope of $p\in\mathbb{K}[\bold{x},y]$, let $w_e$ be any element of $C^*(e)$ which induces a total order on $\mathbb{Q}^n$. Then the list of series roots of $p$ that results from $e$ and $w_e$ when $e$ ranges over all admissible edges of the Newton polytope of $p$ is exhaustive. In particular, $p$ has only finitely many series roots.
\end{Conjecture}

It is natural to ask how the number of series solutions of a polynomial equation can be described.
\begin{Problem}\label{prob:number}
Find a formula for the number of series solutions of a polynomial equation. Can it be related to any statistics of its Newton polytope?
\end{Problem}
The next Proposition provides an answer to Problem~\ref{prob:number} for polynomials of degree $1$.
\begin{Proposition}\label{prop:number}
Let $p,q\in\mathbb{K}[\bold{x}]$ and let $\preceq$ be an additive total order. The series solution of 
\begin{equation*}
p(\bold{x}) - q(\bold{x}) y = 0
\end{equation*}
in~$\mathbb{K}_{\preceq}((\bold{x}))$ depends on $\preceq$ only to the extent of what the leading term $\mathrm{lt}_{\preceq}(q)$ of $q$ with respect to $\preceq$ is. It is 
\begin{equation}\label{eq:geometric}
\frac{p}{\mathrm{lt}_{\preceq}(q)} \sum_{k=0}^\infty \left(1-\frac{q}{\mathrm{lt}_{\preceq}(q)}\right)^k,
\end{equation}
and its support is contained in a shift of the cone generated by the support of $q/\mathrm{lt}_{\preceq}(q)$.
In particular, there is a bijection between the series solutions of the equation and the vertices of the Newton polytope of~$q$. 
\end{Proposition}
\begin{proof}
Consider the Laurent polynomial $q/\mathrm{lt}_{\preceq}(q)$, and let $C$ be the (strictly convex) cone generated by its support. Since $q/\mathrm{lt}_{\preceq}(q)$ is an element of $\mathbb{K}_C[[\bold{x}]]$, and because its constant term is different from zero, it has a multiplicative inverse in $\mathbb{K}_C[[\bold{x}]]$, see~\cite[Theorem 12]{monforte2013formal}. Since~$\mathbb{K}_C[[\bold{x}]]$ is an integral domain~\cite[Theorem 11]{monforte2013formal}, this multiplicative inverse is unique. As a consequence, the series solution of $p(\bold{x})+q(\bold{x})y$ in $\mathbb{K}_{\preceq}((\bold{x}))$ depends on the total order $\preceq$ only to the extent of what $\mathrm{lt}_{\preceq}(q)$ is. The terms of $q$ that appear as the leading term with respect to some additive total order are those whose exponent is a vertex of the Newton polytope of $q$. The number of series solutions is therefore bounded by the number of vertices the Newton polytope has. To see that these numbers are equal, it is sufficient to observe that the series solution in $\mathbb{K}_{\preceq}((\bold{x}))$ is the geometric series expansion of $p/q$ given by~(\ref{eq:geometric}), and to note that for an additive total order for which $q$ has another leading term the corresponding series is different. The statement about the support is obvious from the explicit expression of the series.
\end{proof}

\begin{Example}\label{ex:inv}
The equation
\begin{equation*}
x+y-(1+x+y)z = 0
\end{equation*} 
has three series solutions as the Newton polytope of $1+x+y$ has three vertices:
$(x+y) \sum_{k=0}^\infty (-1)^k(x+y)^k$, $\frac{x+y}{x} \sum_{k=0}^\infty (-1)^k\left(\frac{1+y}{x}\right)^k$ and
$\frac{x+y}{y} \sum_{k=0}^\infty (-1)^k\left(\frac{1+x}{y}\right)^k$. Their supports are contained in shifts of $\mathrm{cone}\{(1,0), (0,1)\}$, $\mathrm{cone}\{ (-1,0), (-1,1) \}$ and $\mathrm{cone}\{ (0,-1), (1,-1)\}$, respectively.
\end{Example}

\section{Effective arithmetic}

We begin with the observation that the sum and product of two algebraic series need not be algebraic.

\begin{Example}\label{ex:welldefinedness}
The geometric series 
\begin{equation*}
\phi_1 = 1 + x + x ^2 + \dots \quad \text{and} \quad \phi_2 = - x^{-1} - x^{-2} - x^{-3} - \dots
\end{equation*}
are both algebraic as they are the series roots of 
\begin{equation*}
p(x,y) := (1-x) y -1,
\end{equation*}
but neither is their sum nor their product as none of them is well-defined.
\end{Example}

The sum and product of two algebraic series $\phi_1$ and $\phi_2$ is algebraic if and only if there is a vector~$w\in\mathbb{R}^n$ that is compatible with both $\phi_1$ and $\phi_2$.
In the following we assume we know a $w\in\mathbb{R}^n$ that induces a total order~$\preceq$ such that $\phi_1,\phi_2 \in\mathbb{K}_{\preceq}((\bold{x}))$, and discuss how to determine finite encodings for $\phi_1 + \phi_2$ and $\phi_1 \phi_2$, given encodings of $\phi_1$ and $\phi_2$.

Assume that $\phi_1$ and $\phi_2$ are two series given by $(p_1, w, q_1)$ and $(p_2,w,q_2)$. An annihilating polynomial $p$ for $\phi_1 + \phi_2$ can be derived from annihilating polynomials $p_1$ and~$p_2$ of $\phi_1$ and $\phi_2$ by computing a generator of the elimination ideal of 
\begin{equation*}
\langle p_1(\bold{x},y_1),\ p_2(\bold{x},y_2),\  y_1+y_2-y_3\rangle \cap \mathbb{K}(\bold{x})[y_3].
\end{equation*}
Whether a series root of $p$ represented by $(p,w,q)$ equals $\phi_1 + \phi_2$ can be decided by computing the truncations $\tilde{q}_1$ and~$\tilde{q}_2$ of $\phi_1$ and $\phi_2$ up to order $\mathrm{ord}(q,w)$, where 
\begin{equation*}
\mathrm{ord}(q,w) := \min \{\alpha\cdot w \;|\; \alpha \in\mathrm{supp}(q)\}.
\end{equation*}
If $\tilde{q}_1 + \tilde{q}_2$ does not equal $q$ when ordered with respect to $w$, the series represented by $(p,w,q)$ does not equal $\phi_1+\phi_2$. But if it does, then $(p,w,q)$ is a finite encoding of $\phi_1+\phi_2$, and by assumption we can be sure that we find a finite encoding in this way.

Similarly, an annihilating polynomial $p$ for $\phi_1\phi_2$ can be determined by computing a generator of the elimination ideal
\begin{equation*}
\langle p_1(\bold{x},y_1),\ p_2(\bold{x},y_2),\ y_1y_2-y_3\rangle \cap \mathbb{K}(\bold{x})[y_3].
\end{equation*}
To find a representation $(p,w,q)$ of $\phi_1\phi_2$ compute the truncation $\tilde{q}_1$ of $\phi_1$ up to order~$\mathrm{ord}(q,w) -  w\cdot \mathrm{lexp}_w(\phi_2)$ and the truncation $\tilde{q}_2$ of $\phi_2$ up to order $\mathrm{ord}(q,w) - w \cdot \mathrm{lexp}_w(\phi_1)$, where $\mathrm{lexp}_w(\phi_i)$ denotes the leading exponent of $\phi_i$ with respect to the total order induced by $w$. If $q$ does not equal the sum of the first terms of $\tilde{q}_1 \tilde{q}_2$ when ordered with respect to $w$, the series represented by~$(p,w,q)$ does not equal $\phi_1\phi_2$. However, if it does, then $(p,w,q)$ is a finite encoding of $\phi_1 \phi_2$, and by assumption we can be sure that we find a finite encoding in this way.

Other closure properties for algebraic series such as taking multiplicative inverses or derivatives can be performed similarly. We just refer to~\cite[Theorem 6.3]{kauers2011concrete} for how to compute the corresponding annihilating polynomials.

\section*{Acknowledgement}
Thanks go to the Austrian FWF, the Johannes Kepler University Linz and the state of Upper Austria which supported this work with the grants F5004, P31571-N32 and LIT-2022-11-YOU-214, respectively. Thanks go also to the author's colleagues at the Johannes Kepler University Linz and the Austrian Academy of Sciences, in particular to Rapha\"{e}l Pag\`{e}s, for the discussions we had. Thanks to the reviewers for their careful reading and their comments.  

\bibliographystyle{plain}
\bibliography{newtonPuiseux}

\end{document}